\documentclass{amsart}
\usepackage{amssymb}
\usepackage{amsthm}
 \newtheorem{lem}{Lemma}

\theoremstyle{definition} 

\theoremstyle{remark} \newtheorem{rem}{Remark}
\begin{document}

\title{Geometry of totally real Galois fields \\ of degree 4}
\author{Yury Kochetkov}
\date{}
\email{yuyuk@prov.ru}

\begin{abstract}{We will consider a totally real Galois field $K$
of degree 4 as the linear coordinate space $\mathbb{Q}^4\subset\mathbb{R}^4$.
An element $k\in K$ is called strictly positive, if all its conjugates are
positive. The set of strictly positive elements is a convex cone in $K$. The
convex hull of strictly positive integral elements is a convex subset of this
cone and its boundary $\Gamma$ is an infinite union of 3-dimensional
polyhedrons. The group $U$ of strictly positive units acts on $\Gamma$: the
action of a strictly positive unit permutes polyhedrons. Fundamental domains of
this action are the object of study in this work. We mainly present some
interesting examples.}\end{abstract}

\maketitle

\section{Introduction}

Let $K$ be a cubic totally real Galois field, defined by polynomial
$p=x^3+ax+b\in\mathbb{Q}[x]$. Elements of $K$ will be written in the form
$kx^2+lx+m$, $k,l,m\in\mathbb{Q}$. This notation allows us to identify $K$ with
3-dimensional space $\mathbb{Q}^3\subset \mathbb{R}^3$. . An element $k\in K$
is called strictly positive, if all its conjugates are positive, i.e. if
$x_1,x_2,x_3$ are (real) roots of $p$ then an element $kx^2+lx+m$ is strictly
positive, if
$$\left\{\begin{array}{l}kx_1^2+lx_1+m>0\\ kx_2^2+lx_2+m>0\\ kx_3^2+lx_3+m>0
\end{array}\right.$$ These three conditions define a convex cone $C$ in
$\mathbb{R}^3$. Let $O\subset C$ be the set of strictly positive integral
elements of $K$, $\overline{O}$ --- its convex hull (an infinite polyhedral
subset of $C$) and $\Gamma$ be the boundary of this convex hull --- an infinite
polygonal 2-dimensional complex in $\mathbb{R}^3$. Let $U$ be the group of
strictly positive units (a free Abelian group of rank 2). The action of a
positive unit on $\Gamma$ induces a permutation of its faces. The fundamental
domain of the action of $U$ on $\Gamma$ is a finite union of polygons with
pairwise identified sides --- a torus. In [1] the structure of this polygonal
complex was studied.

In this work we will study a 4-dimensional analogue of this problem.

\section{Totally real Galois fields of degree 4}

Let $K$  be a totally real Galois field of degree 4 defined by a polynomial
$p\in \mathbb{Q}[x]$. Elements of the field will be written in the form
$kx^3+lx^2+mx+n$, $k,l,m,n\in \mathbb{Q}$, and the field will be considered as
4-dimensional space $\mathbb{Q}^4\subset \mathbb{R}^4$ with coordinates
$k,l,m,n$. By $O$ will be denoted the set of integral strictly positive
elements from $K$: an element $kx^3+lx^2+mx+n$ is called strictly positive if
$$\left\{\begin{array}{l}kx_1^3+lx_1^2+mx_1+n>0\\ kx_2^3+lx_2^2+mx_2+n>0\\
kx_3^3+lx_3^2+mx_3+n>0\\ kx_4^3+lx_4^2+mx_4+n>0 \end{array}\right.$$ where
$x_1,x_2,x_3,x_4$ are (real) roots of $p$. These four conditions define a
convex cone $C\subset\mathbb{R}^4$. The convex hull $\overline{O}$ of the set
$O$ is an infinite polyhedral 4-dimensional subset of $C$. Its boundary
$\Gamma$ is an infinite 3-dimensional polyhedral complex. Let $U\subset O$ be
the group of strictly positive units --- a free Abelian group of rank 3. The
action of a positive unit on $\Gamma$ induces a permutation of its
3-dimensional polyhedra. We will study fundamental domain of action of $U$ on
$\Gamma$. This domain --- a finite union of 3-dimensional polyhedrons with
pairwise identified faces is homeomorphic to 3-dimensional torus [2].

Galois group $G$ of our field $K$ is either $\mathbb{Z}_4$, or
$\mathbb{Z}_2\oplus\mathbb{Z}_2$. In both cases $G$ contains an invariant
subgroup of order 2, so $K$ contains a subfield $L$ of degree 2. $L$ is a
decomposition field of a second degree polynomial with rational coefficients
and positive discriminant $d$. Hence, $K$ is a second degree algebraic
extension of $L$, defined by polynomial $r\in\mathbb{Q}(\sqrt d)[x]$.

Let $r=x^2+2sx+t$, $s,t\in\mathbb{Q}(\sqrt d)$, then roots
$x_1,x_2$ of $r$ are of the form
$$x_{1,2}=-s\pm \sqrt{s^2-t}\in K\Rightarrow \sqrt{s^2-t}\in K.$$
If $s^2-t=m+n\sqrt d$, $m,n\in \mathbb{Q}$, then $\sqrt{m+n\sqrt d}\in K$, i.e.
all conjugates $\pm \sqrt{m\pm n\sqrt d}$ belong to $K$. Thus, the field $K$ is
the decomposition field of the biquadratic polynomial
\begin{multline*}q=(x-\sqrt{m+n\sqrt d})(x+\sqrt{m+n\sqrt d})(x-\sqrt{m-n\sqrt d})
(x+\sqrt{m-n\sqrt d})=\\=x^4-2mx^2+m^2-n^2d=x^4-2ax^2+b.\end{multline*} Such
biquadratic polynomial defines a totally real Galois field if: 1) all its roots
are real, i.e. if $a>0$, $b>0$ and $d=a^2-b>0$; 2) a root $y_1=\sqrt{a+\sqrt
d}$ of $q$ polynomially generates all other roots of $q$, in particular, the
root $y_2=\sqrt{a-\sqrt d}$, which can represented only in the form
$y_2=ky_1^3+ly_1$. This gives us two possibilities:
\begin{itemize}
    \item either $ky_2^3+ly_2=-y_1$, then $G=\mathbb{Z}_4$, $a^2b-b^2=bd=c^2$,
    $l=(2a^2-b)/c=(a^2+d)/c$ and $k=-a/c$;
    \item or $ky_2^3+ly_2=y_1$, then $G=\mathbb{Z}_2\oplus\mathbb{Z}_2$,
    $b=c^2$, $l=2a/c$ and $k=-1/c$.
\end{itemize}

\begin{rem} If $b$ is a full square and $a^2b-b^2$ is also a full
square, then polynomial $x^4-2ax^2+b$ is reducible. \end{rem}

In what follows we'll assume that $2a$ and $b$ are positive integers. Let us
describe at first the structure of the fundamental domain for fields with
smallest $a$ and $b$: for the field $K_1=\mathbb{Q}[x]/(x^4-4x^2+1)$ and for
the field $K_2=\mathbb{Q}[x]/(x^4-4x^2+2)$ (polynomial $x^4-3x^2+1$ is
reducible).

\section{The field $K_1=\mathbb{Q}[x]/(x^4-4x^2+1)$}

Galois group here is $\mathbb{Z}_2\oplus\mathbb{Z}_2$. Transformations
$x\mapsto x^3-4x$ and $x\mapsto -x$ are generators of Galois group. Elements
$kx^3+lx^2+mx+n$, $k,l,m,n\in \mathbb{Z}$ are integral.

Hyperplane $k+l+n-1=0$ is a support hyperplane for the set $O$. It
contains 9 integral strictly positive elements:
$$\begin{array}{lll} A_1=(-2,0,6,3),&A_2=(-2,3,2,0),&A_3=(-1,0,3,2),\\
A_4=(-1,0,4,2),&A_5=(-1,1,2,1),&A_6=(-1,2,0,0),\\ A_7=(-1,2,1,0),&
A_8=(0,0,0,1),&A_9=(0,1,0,0).\end{array}$$ Elements $A_1,A_2,A_4,A_6,A_8$ and
$A_9$ are units. Element $A_3$ is a midpoint of segment $[A_1,A_8]$ and has
norm 4; element $A_7$ is a midpoint of segment $[A_2,A_9]$ and has norm 4;
element $A_5$ is a midpoint of segment $[A_4,A_6]$ and has norm 9. The convex
hull of these points in our hyperplane is an octahedron:
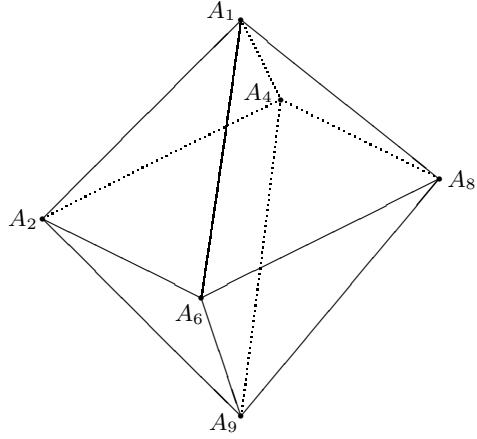
\[\begin{picture}(170,170) \put(10,85){\circle*{2}} \put(70,55){\circle*{2}}
\put(85,10){\circle*{2}} \put(85,160){\circle*{2}} \put(100,130){\circle*{2}}
\put(160,100){\circle*{2}} \put(10,85){\line(1,1){75}}
\put(10,85){\line(2,-1){60}} \put(10,85){\line(1,-1){75}}
\put(70,55){\line(1,-3){15}} \put(70,55){\line(2,1){90}}
\put(85,10){\line(5,6){75}} \put(85,160){\line(5,-4){75}}
\qbezier(70,55)(78,108)(85,160) \qbezier[50](10,85)(55,108)(100,130)
\qbezier[30](100,130)(130,115)(160,100) \qbezier[20](85,160)(93,145)(100,130)
\qbezier[60](85,10)(93,70)(100,130) \put(-3,82){\small $A_2$}
\put(163,97){\small $A_8$} \put(73,5){\small $A_9$} \put(72,161){\small $A_1$}
\put(60,46){\small $A_6$} \put(86,130){\small $A_4$} \end{picture}\] Faces of
this octahedron are not identified by the action of group $U$. Hence, the
construction of fundamental domain is not finished.

Indeed, hyperplane $4k+5l/2+m+n-1=0$ is also a support hyperplane for the set
$O$. It contains 5 integral strictly positive elements:
$$A_1,\,\,A_3,\,\,A_6,\,\,A_8\text{ and also an element } A_{10}=(0,0,-1,2),$$
which is a unit. The convex hull of these points is a tetrahedron, which is
glued to our octahedron (by the triangular face $A_1A_6A_8$). In thus obtained
polyhedron some faces (not all!) are identified:
\begin{enumerate}
    \item multiplication by the unit $A_9$ identifies the face $A_1A_8A_{10}$
    with the face $A_2A_6A_9$: $A_1\to A_2$, $A_8\to A_9$, $A_{10}\to A_6$;
    \item multiplication by the unit $A_4$ identifies the face $A_6A_8A_{10}$
    with the face $A_1A_2A_4$: $A_6\to A_2$, $A_8\to A_4$, $A_{10}\to A_1$;
    \item multiplication by the unit $A_{10}$ identifies the face $A_4A_8A_9$
    with the face $A_1A_6A_{10}$: $A_4\to A_1$, $A_8\to A_{10}$,
    $A_9\to A_6$.
\end{enumerate} Faces $A_1A_4A_8$, $A_1A_2A_6$, $A_6A_8A_9$ ¨ $A_2A_4A_9$ are
"free". Hence, the construction of fundamental domain is not finished.

Indeed, hyperplane $k+l/2+n-1=0$ is also a support hyperplane for the set $O$.
It contains 5 integral strictly positive elements:
$$A_1,\,\,A_3,\,\,A_4,\,\,A_8\text{ and also an element }
A_{11}=(-4,-2,15,8),$$ which is a unit. The convex hull of these points is a
tetrahedron, which is glued to our octahedron (by the triangular face
$A_1A_4A_8$). In thus obtained polyhedron all its 2-dimensional faces are
pairwise identified:
\begin{enumerate}
    \item multiplication by the unit $A_9$ identifies the face $A_1A_8A_{11}$
    with the face $A_2A_4A_9$: $A_1\to A_2$, $A_8\to A_9$, $A_{11}\to A_4$;
    \item multiplication by the unit $A_6$ identifies the face $A_4A_8A_{11}$
    with the face $A_1A_2A_6$: $A_4\to A_2$, $A_8\to A_6$, $A_{11}\to A_1$;
    \item multiplication by the unit $A_{11}$ identifies the face $A_6A_8A_9$
    with the face $A_1A_4A_{11}$: $A_6\to A_1$, $A_8\to A_{11}$,
    $A_9\to A_4$.
\end{enumerate} The construction of the fundamental domain is finished.

\section{The field $K_2=\mathbb{Q}[x]/(x^4-4x^2+2)$}

Galois group here is $\mathbb{Z}_4$. Transformation $x\mapsto x^3-3x$ is a
generator of Galois group. Elements $kx^3+lx^2+mx+n$, $k,l,m,n\in \mathbb{Z}$
are integral.

Hyperplanes
$$\begin{array}{lll} \Pi_1:\,k+2l+m+2n-2=0, && \Pi_2:\,
3k+4l+2m+4n-4=0,\\ \Pi_3:\,k+2l+2n-2=0, && \Pi_4:\, k+4l-m+4n-4=0,\\
\Pi_5:\,k-2l+m-2n+2=0, && \Pi_6:\, 3k-4l+2m-4n+4=0,\\
\Pi_7:\, k-2l-2n+2=0,&& \Pi_8:\, k-4l-m-4n+4=0\end{array}$$ are support
hyperplanes for the set $O$. Each hyperplane contains 5 integral elements.
Three elements $A=(0,0,0,1)$, $B=(0,-2,0,-1)$ and $(A+B)/2$ belong to each
hyperplane. Thus, the convex hull of the set $O\cap\Pi_i$ is an tetrahedron.

Besides $A$, $B$ and $(A+B)/2$, hyperplane $\Pi_1$ contains elements
$C=(1,0,-3,2$ and $D=(0,1,-2,1)$, hyperplane $\Pi_2$ --- $D$ and
$E=(-2,4,1,-2)$, hyperplane $\Pi_3$ --- $E$ and $F=(-2,3,2,-1)$, hyperplane
$\Pi_4$ --- $F$ and $G=(-1,0,3,2)$, hyperplane $\Pi_5$ --- $G$ and
$H=(0,1,2,1)$, hyperplane $\Pi_6$ --- $H$ and $K=(2,4,-1,-2)$, hyperplane
$\Pi_7$ --- $K$ and $L=(2,3,-2,-1)$, hyperplane $\Pi_8$ --- $L$ and $C$. The
union of these tetrahedrons is something like two octagonal pyramids
$ACDEFGHKL$ and $BCDEFGHKL$ with common base $CDEFGHKL$. Here elements $A$,
$B$, $D$, $F$, $H$ and $L$ are units and elements $C$, $E$, $G$ and $K$ have
norm 2.

In thus constructed polyhedron some faces (not all!) are identified:
\begin{itemize}
    \item the multiplication by unit $H$ identifies the face $ACD$ with the
    face $BHK$: $A\to H$, $C\to K$, $D\to B$;
    \item the multiplication by unit $L$ identifies the face $AFG$ with the
    face $BKL$: $A\to L$, $F\to B$, $G\to K$;
    \item the multiplication by unit $D$ identifies the face $AGH$ with the
    face $BDE$: $A\to D$, $G\to E$, $H\to B$;
    \item the multiplication by unit $F$ identifies the face $ALC$ with the
    face $BEF$: $A\to F$, $C\to E$, $L\to B$.
\end{itemize} Faces $ADE$, $AEF$, $AHK$, $AKL$, $BCD$, $BFG$, $BGH$ and $BLC$
are "free". The construction of the fundamental domain is not finished.

$3k+2l+m+n-1=0$ also is a support hyperplane for the set $O$. It contains 7
integral strictly positive elements: $A$, $D$, $E$, $F$, $M=(0,0,-1,2)$,
$N=(-2,4,0,-1)$ and $(-1,2,0,0)=(A+N)/2=(D+F)/2=(E+M)/2$. The convex hull of
these points is an octahedron, where $(A,N)$, $(D,F)$ and $(E,M)$ --- pairs of
opposite vertices. Here element $N$ is a unit and element $M$ has norm 2.

In thus obtained polyhedron --- two octagonal pyramids with a common base and
octahedron glued to them by the faces $ADE$ and $AFE$, the following faces are
identified:
\begin{itemize}
    \item the multiplication by unit $N$ identifies the face $AHK$ with the
    face $EFN$: $A\to N$, $H\to F$, $K\to E$;
    \item the multiplication by unit $N$ identifies the face $AKL$ with the
    face $NDE$: $A\to N$, $K\to E$, $L\to D$;
    \item the multiplication by unit $H$ identifies the face $ADM$ with the
    face $BGH$: $A\to H$, $D\to B$, $M\to G$;
    \item the multiplication by unit $L$ identifies the face $AFM$ with the
    face $BCL$: $A\to L$, $F\to B$, $M\to C$;
    \item the multiplication by unit $B^{-1}F$ identifies the face $BCD$ with the
    face $FMN$: $B\to F$, $C\to M$, $D\to N$;
    \item the multiplication by unit $B^{-1}D$ identifies the face $BFG$ with the
    face $DMN$: $B\to D$, $F\to N$, $G\to M$.
\end{itemize} The construction of the fundamental domain is finished.

\section{The simplest fundamental domains in the case of cyclic Galois group}

The combinatorics of a fundamental domain can be quite complex. The natural
question here is: how simple it can be? It turns out that there exists an one
parameter family of fields with a simple domains and also one field apart. At
first we will study the special field.

\subsection{The field $K=\mathbb{Q}[x]/(x^4-15x^2+45)$} The transformation
$x\mapsto x^3/3-3x$ is a generator of Galois group. Elements of the form
$$(k,l,m,n)+(\frac i6,0,0,\frac i2)+(0,\frac j6,\frac j2,\frac j2),\quad
k,l,m,n\in\mathbb{Z},\quad i,j=0,\ldots,5,$$ are integral.

Hyperplane $\Pi: 3k+6l+m+n=1$ is a support hyperplane for the set $O$. It
contains 8 integral strictly positive elements --- all of them are units:
$$\begin{array}{l} A=(0,0,0,1),\, B=\Bigl(\dfrac 16,-\dfrac 13,-2,\dfrac 92
\Bigr),\, C=\Bigl(\dfrac 13,-\dfrac 12,-\dfrac 72,\dfrac{13}{2}\Bigr),
\,D=\Bigl(\dfrac 16,-\dfrac 16,-\dfrac 32, 3\Bigr), \\ \\ A_1=\Bigl(0,\dfrac
13,0,-1\Bigr),\,B_1=\Bigl(0,\dfrac 16,-\dfrac 12,\dfrac 12\Bigr),
\,C_1=\Bigl(\dfrac 16,\dfrac 16,-\dfrac 32,1\Bigr),\, D_1=\Bigl(\dfrac
16,\dfrac 13,-1,-\dfrac 12\Bigr).\end{array}$$ The convex hull of the set
$O\cap\Pi$ is a decahedron with 2 quadrangular faces $ABCD$ and $A_1B_1C_1D_1$
--- bases, and 8 triangular faces. Bases are parallelograms, that are parallel
to each other. In the picture below is presented the pattern of triangular
faces of this "prism":
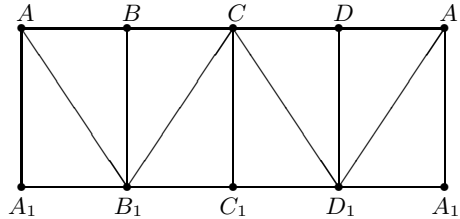
\[\begin{picture}(160,115) \multiput(0,40)(40,0){5}{\circle*{3}}
\multiput(0,100)(40,0){5}{\circle*{3}} \put(0,40){\line(1,0){160}}
\put(0,100){\line(1,0){160}} \multiput(0,40)(40,0){5}{\line(0,1){60}}
\put(-2,103){\small $A$} \put(38,103){\small $B$} \put(78,103){\small $C$}
\put(118,103){\small $D$} \put(158,103){\small $A$} \put(-5,30){\small $A_1$}
\put(35,30){\small $B_1$} \put(75,30){\small $C_1$} \put(115,30){\small $D_1$}
\put(155,30){\small $A_1$} \put(0,100){\line(2,-3){40}}
\put(40,40){\line(2,3){40}} \put(80,100){\line(2,-3){40}}
\put(120,40){\line(2,3){40}}
\end{picture}\] Faces of $ABCDA_1B_1C_1D_1$ are identified in the following way:
\begin{itemize}
    \item the multiplication by unit $A_1$ identifies the base $ABCD$ with the
    face $A_1B_1C_1D_1$: $A\to A_1$, $B\to B_1$, $C\to C_1$, $D\to D_1$;
    \item the multiplication by unit $D$ identifies the face $AA_1B_1$ with the
    face $CDD_1$: $A\to D$, $A_1\to D_1$, $B_1\to C$;
    \item the multiplication by unit $D_1$ identifies the face $ABB_1$ with the
    face $CC_1D_1$: $A\to D_1$, $B\to C$, $B_1\to C_1$;
    \item  the multiplication by unit $B$ identifies the face $AA_1D_1$ with the
    face $BCB_1$: $A\to B$, $A_1\to B_1$, $D_1\to C$;
    \item the multiplication by unit $B_1$ identifies the face $ADD_1$ with the
    face $CB_1C_1$: $A\to B_1$, $D\to C$, $D_1\to C_1$.
\end{itemize} The construction of the fundamental domain is finished.

\subsection{Fields $K_n$} Let $p_n=x^4-(n^2+4)x^2+n^2+4$, where $n$ is an odd
positive integer, and let $K_n$ be the field, defined by $p_n$.

\begin{lem} Polynomials $p_n$ are irreducible. \end{lem}

\begin{proof} If $p_n$ is reducible, then either it has a positive integral
root, or it has an irreducible factor of degree 2.

Let $a$ be a positive integral root of $p_n$: $a^4-(n^2+4)a^2+n^2+4=0$. Each
prime factor of $a$ is a factor of the number $n^2+4$ and vice versa. Let $q$
be such prime factor, $\alpha>0$ be an exponent of $a$ and $\beta>0$ be an
exponent of $n^2+4$. In the set of 3 numbers $\{4\alpha,2\alpha+\beta,\beta\}$
cannot be exactly one minimal, thus, $4\alpha=\beta$. As this is true for each
prime factor of $a$, then $n^2+4$ is a full square --- a contradiction.

Let $p_n$ be a product of two irreducible polynomials of degree 2 with integral
coefficients, then these factors are $x^2-ax+b$ and $x^2+ax+b$:
$$(x^2-ax+b)(x^2+ax+b)=x^4-(n^2+4)x^2+n^2+4.$$ But then again $n^2+4$ is a full
square. \end{proof}

\begin{rem} Fields $K_n$ are not pairwise different. Indeed, fields
$K_1$ and $K_{11}$ coincide:  if $y=-3x^3+5x\in K_1$, then $y^4-125y^2+125=0$.
\end{rem}

\subsection{The field $K_1\simeq\mathbb{Q}[x]/(x^4-5x^2+5)$} The transformation
$x\mapsto x^3-3x$ is a generator of Galois group. Elements of the form
$kx^3+lx^2+mx+n$, $k,l,m,n\in \mathbb{Z}$ are integral.

Hyperplane $\Pi: 2k+2l+m+n=1$ is a support hyperplane for the set $O$. $\Pi$
contains 10 integral strictly positive elements:
$$\begin{array}{ll}A=(0,0,0,1),&A_1=(0,1,0,-1),\\B=(0,0,-1,2),&B_1=(-1,2,1,-2),\\
C=(2,-2,-8,9),&C_1=(0,1,-2,1),\\D=(2,-2,-7,8),&D_1=(1,0,-3,2),\\
E=(1,-1,-4,5),&E_1=(0,1,-1,0).\end{array}$$The convex hull of this points in
$\Pi$ is a decahedron, combinatorially the same, as in subsection 5.1. The only
difference is that centers $E$ and $E_1$ of bases $ABCD$ and $A_1B_1C_1D_1$,
respectively, also are integral and strictly positive.

\subsection{The field $K_3\simeq\mathbb{Q}[x]/(x^4-13x^2+13)$} The transformation
$x\mapsto x^3/3-11x/3$ is a generator of Galois group. Elements of the form
$$(k,l,m,n)+\left(\frac i3,0,\frac i3,0\right)+\left(0,\frac j3,0,\frac j3\right),
\quad k,l,m,n\in \mathbb{Z},\quad i,j=0,1,2,$$ are integral.

Hyperplane $\Pi: 2k+2l+m+n=1$ is a support hyperplane for the set $O$. $\Pi$
contains 36 integral strictly positive elements. The convex hull of this points
in $\Pi$ is a decahedron, combinatorially the same, as in subsection 5.1.
Besides vertices $A,B,C,D,A_1,B_1,C_1,D_1$ of decahedron the intersection
$\Pi\cap O$ contains:
\begin{enumerate}
    \item centers of bases $ABCD$ and $A_1B_1C_1D_1$;
    \item two points on each segment $[A,A_1]$, $[B,B_1]$, $[C,C_1]$ and
    $[D,D_1]$, which divide each segment into 3 equal parts;
    \item 18 points inside the decahedron.
\end{enumerate}

\subsection{General case} Let assume that a field $K_n$ is not isomorphic to
any field $K_m$, $m<n$. Put $n=2a+1$, then transformation
$$x\mapsto \frac{x^3}{2a+1}-\frac{4a^2+4a+3}{2a+1}\,x$$ is a generator of
Galois group. Elements of the form
\begin{multline*}(k,l,m,n)+i\left(\frac{1}{2a+1},0,\frac{2a-1}{2a+1},0\right)+j
\left(0,\frac{1}{2a+1},0,\frac{2a-1}{2a+1}\right),\\ \\
k,l,m,n\in\mathbb{Z},\quad i,j=0,\ldots,2a,\end{multline*} are integral.

Hyperplane $\Pi: 2k+2l+m+n=1$ is a support hyperplane for the set $O$. The
convex hull of the set $\Pi\cap O$ in $\Pi$ is the same decahedron with
vertices $A,B,C,D,A_1,B_1,C_1,D_1$ and bases --- parallelograms $ABCD$ and
$A_1B_1C_1D_1$. In internal coordinates $k,l,m$ of the hyperplane $\Pi$ bases
lie in parallel planes $k+l=0$ and $k+l=1$. All vertices are units: $A=1$,
$$B=\dfrac{2(a+1)x^3-2(a+1)x^2-(8a^3+16a^2+16a+7)x+
8a^3+16a^2+18a+8}{2a+1},$$\par\medskip $$C=2x^3-2x^2-(8a^2+8a+8)x+8a^2+8a+9,$$
\par\medskip
$$D=\dfrac{2ax^3-2ax^2-(8a^3+8a^2+8a+1)x+ 8a^3+8a^2+10a+2}{2a+1},$$
\par\medskip
$$A_1=x^2-1,\quad B_1=\dfrac{x^3+2ax^2-(2a+3)x+2}{2a+1},$$ \par\medskip
$$C_1=x^2-2x+1,\quad D_1=\dfrac{-x^3+(2a+2)x^2-(2a-1)x-2}{2a+1}.$$
The multiplication by unit $A_1$ identifies bases $ABCD$ and $A_1B_1C_1D_1$.
Other 8 triangular faces of the "prism" are pairwise identified in the
following way:
$$\begin{array}{ll}\triangle AA_1B_1\sim \triangle CDD_1,& \triangle ABB_1\sim \triangle
CC_1D_1,\\ \triangle BB_1C\sim \triangle D_1AA_1,& \triangle B_1CC_1 \sim
\triangle DD_1A.\end{array}$$ Besides vertices, the intersection $\Pi\cap O$
contains centers of bases, $2a$ points on each segment $[A,A_1]$, $[B,B_1]$,
$[C,C_1]$ and $[D,D_1]$, which divide each segment into $2a+1$ equal parts, and
points inside the decahedron.

\begin{rem} Integers of the field $K_{11}$ are of the form
$$(k,l,m,n)+i\left(\frac{1}{275},0,\frac{3}{11},0\right)+j\left(0,\frac{1}{55},0,
\frac{4}{11}\right).$$ Therefore, the described above "prism" contains
\emph{several} fundamental decahedrons of the field $K_{11}$. \end{rem}

\section{Simple fundamental domains in the case of Galois group $\mathbb{Z}_2\oplus
\mathbb{Z}_2$}

In the case of Klein group we don't know fields, where fundamental domain
consists of one 3-dimensional face of $\Gamma$, but there are examples, where
this domain consists of two.

\subsection{The field $K=\mathbb{Q}[x]/(x^4-9x^2+9)$} Transformations $x\mapsto
x^3/3-3x$ and $x\mapsto -x$ are generators of Galois group. An element
$kx^3+lx^2+mx+n\in K$ is integral, if $3k,3l,m,n\in \mathbb{Z}$.

Hyperplane $\Pi_1:\,6k+3l+m+n-1=0$ is a support hyperplane for the
set $O$. It contains 14 integral strictly positive elements. This
elements belong to 2 parallel 2-dimensional planes in $\Pi_1$:
$P_0:\,9k+3l+m=0$ and $P_1:\,9k+3l+m=1$ (seven in each). Elements
$A=(0,0,0,1)$, $B=(0,1/3,-1,1)$, $C=(-1/3,4/3,-1,0)$,
$D=(-2/3,2,0,1)$, $E=(-2/3,5/3,1,-1)$, $F=(-1/3,2/3,1,0)$ and
$G=(-1/3,1,0,0$ belong to the plane $P_0$. Points $A,B,C,D,E$ and
$F$ are vertices of centrally-symmetric hexagon; $A$, $C$, $D$ and
$F$ are units; $B$ and $E$ have norm 4; $G$ is the center of
hexagon and has norm 9.

Elements $A_1=(0,1/3,0,0)$, $B_1=(-1/3,4/3,0,-1)$,
$C_1=(-4/3,4,1,-4)$, $D_1=(-2,17/3,2,-6)$, $E_1=(-5/3,14/3,2,-5)$,
$F_1=(-2/3,2,1,-2)$ and $G=(-1,3,1,-3$ belong to the plane $P_1$.
Points $A_1,B_1,C_1,D_1,E_1$ and $F_1$ are vertices of
centrally-symmetric hexagon; $A_1$, $C_1$, $D_1$ and $F_1$ are
units; $B_1$ and $E_1$ have norm 4; $G_1$ is the center of hexagon
and has norm 9.

Polyhedron $M_1$ --- the convex hull in $\Pi_1$ of the set
$O\cap\Pi_1$, is a "prism" with two parallel hexagon bases and ten
"side" faces: $ABA_1$, $A_1B_1CB$, $B_1C_1C$, $C_1D_1C$, $CDD_1$,
$DED_1$, $D_1E_1FE$, $E_1F_1F$, $FAA_1$, $F_1A_1F$. The
multiplication by unit $A_1$ identifies bases. Also
\begin{itemize}
    \item the multiplication by unit $F$ identifies the face $ABA_1$
    with the face $E_1F_1F$: $A\to F$, $B\to E_1$, $A_1\to F_1$;
    \item the multiplication by unit $A_1^{-1}F$ identifies the
    face $A_1B_1CB$ with the face $D_1E_1FE$: $A_1\to F$, $B_1\to
    E_1$, $C\to D_1$, $B\to E$;
    \item the multiplication by unit $C^{-1}D$ identifies the face
    $B_1C_1C$ with the face $DED_1$: $B_1\to E$, $C_1\to D_1$, $C\to
    D$.
\end{itemize}

Hyperplane $\Pi_2:\,3k+3l+n-1=0$ also is a support hyperplane for
the set $O$. The intersection $\Pi_2\cap O$ contains 4 elements:
$A$, $A_1$, $F$ and a unit $H=(-1/3,0,2,2)$. Polyhedron $M_2$ ---
the convex hull in $\Pi_2$ of the set $O\cap\Pi_2$, is a
tetrahedron, glued to triangular "side" face of $M_1$. The "side"
surface of $M_1\cup M_2$ is presented below:
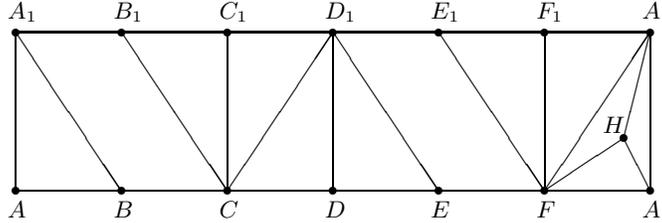
\[\begin{picture}(240,90) \put(0,15){\line(1,0){240}}
\put(0,75){\line(1,0){240}} \put(0,15){\line(0,1){60}}
\put(240,15){\line(0,1){60}} \multiput(0,15)(40,0){7}{\circle*{3}}
\multiput(0,75)(40,0){7}{\circle*{3}} \put(0,75){\line(2,-3){40}}
\put(40,75){\line(2,-3){40}} \put(80,15){\line(0,1){60}}
\put(80,15){\line(2,3){40}} \put(120,15){\line(0,1){60}}
\put(120,75){\line(2,-3){40}} \put(160,75){\line(2,-3){40}}
\put(200,15){\line(0,1){60}} \put(200,15){\line(2,3){40}}
\put(230,35){\circle*{3}} \put(200,15){\line(3,2){30}}
\put(230,35){\line(1,-2){10}} \put(230,35){\line(1,4){10}} \put(-3,5){\small
$A$} \put(37,5){\small $B$} \put(77,5){\small $C$} \put(117,5){\small $D$}
\put(157,5){\small $E$} \put(197,5){\small $F$} \put(237,5){\small $A$}
\put(-3,80){\small $A_1$} \put(37,80){\small $B_1$} \put(77,80){\small $C_1$}
\put(117,80){\small $D_1$} \put(157,80){\small $E_1$} \put(197,80){\small
$F_1$} \put(237,80){\small $A_1$} \put(222,37){\small $H$}
\end{picture}\]
Faces $C_1D_1C$, $CDD_1$, $F_1A_1F$, $AA_1H$, $A_1FH$ and $FAH$
are pairwise identified in the following way:
\begin{itemize}
    \item the multiplication by unit $C^{-1}A$ identifies the face
    $C_1D_1C$ with the face $AA_1H$: $C_1\to A_1$, $D_1\to H$,
    $C\to A$;
    \item the multiplication by unit $C^{-1}A_1$ identifies the
    face $CDD_1$ with the face $A_1FH$: $C\to A_1$, $D\to H$,
    $D_1\to F$;
    \item the multiplication by unit $A_1^{-1}A$ identifies the
    face $F_1A_1F$ with the face $FAH$: $F_1\to F$, $A_1\to A$,
    $F\to H$.
\end{itemize} The construction of the fundamental domain is
finished.

\subsection{The field $K=\mathbb{Q}[x]/(x^4-25x^2+25)$} Transformations
$x\mapsto x^3/5-5x$ and $x\mapsto -x$ are generators of Galois group. An
element $kx^3+lx^2+mx+n\in K$ is integral, if $5k,5l,m,n\in \mathbb{Z}$.
Hyperplanes $\Pi_1:\,5k+5l+m+n=1$ and $\Pi_2:\,-75k+20l-3m+n=1$ are support
hyperplanes for the set $O$. The intersection $\Pi_1\cap O$ contains 14
elements and the intersection $\Pi_2\cap O$ --- 4. The combinatorial structure
of the fundamental polyhedron here is the same as in the previous example.

\begin{rem} The case of the field $K=\mathbb{Q}[x]/(x^4-49x^2+49)$ is
significantly more complex. \end{rem}

\end{document}